\documentclass[preprint,12pt]{elsarticle}
\usepackage[top=2.5cm, bottom=2.5cm,left=2.5cm, right=2.5cm]{geometry}
\usepackage{amssymb}
\usepackage{epsfig}
\usepackage{amsfonts}
\usepackage{amsmath}
\usepackage{euscript}
\usepackage{amscd}
\usepackage{amsthm}
\usepackage{enumitem}
\usepackage{natbib}
\usepackage{tikz}
\usetikzlibrary{calc}

\DeclareMathAlphabet{\mathpzc}{OT1}{pzc}{m}{it}

\def\Aut{\operatorname{Aut}}
\def\edm{\operatorname{End}} 
\def\hom{\operatorname{Hom}}
\def\spec{\operatorname{Spec}} 
\def\mod{\operatorname{mod}}
\def\max{\operatorname{Max}} 

\def\GL{{\operatorname{GL}}}          %General group
\def\SL{{\operatorname{SL}}}          %Special
\def\O{{\operatorname{O}}}          % orthogonal
\def\SO{{\operatorname{SO}}}          % Special orthogonal
\def\EO{{\operatorname{EO}}}          %Elementary orthogonal

\theoremstyle{plain} 
\newtheorem{theorem}{Theorem}[section]
\newtheorem{lemma}[theorem]{Lemma}
\newtheorem{corollary}[theorem]{Corollary}

\theoremstyle{definition} 
\newtheorem{remark}[theorem]{Remark}

\newtheorem{definition}[theorem]{Definition}
\newtheorem{notation}[theorem]{Notation}

\numberwithin{equation}{section}

\makeatletter
\def\ps@pprintTitle{
\def\@oddfoot{}
}
\long\def\pprintMaketitle{\clearpage
  \iflongmktitle\if@twocolumn\let\columnwidth=\textwidth\fi\fi
  \resetTitleCounters
  \def\baselinestretch{1}%
  \printFirstPageNotes
  \begin{center}%
 \thispagestyle{pprintTitle}%
   \def\baselinestretch{1}%
    \LARGE\@title\par\vskip18pt
    \normalsize\elsauthors\par\vskip10pt
    \footnotesize\itshape\elsaddress\par\vskip24pt
      \begin{quote}
      \begin{center}
     {\bfseries \upshape Dedicated to late Professor Amit Roy }
     \end{center}
      \end{quote}
      \vskip18pt
    \hrule\vskip12pt
     \ifvoid\absbox\else\unvbox\absbox\par\vskip10pt\fi
    \ifvoid\keybox\else\unvbox\keybox\par\vskip10pt\fi
    \hrule\vskip12pt
    \end{center}%
     \gdef\thefootnote{\arabic{footnote}}%
  }
\makeatother

\begin{document}
\begin{frontmatter}
\title{ Extendability of  quadratic modules over a polynomial extension of an equicharacteristic regular local ring}
\author[Ambily]{A. A. Ambily}
\ead{ambily@isibang.ac.in}
\author[Rao]{Ravi A. Rao}
%\cortext[cor]{Corresponding author.}
\ead{ravi@math.tifr.res.in}
\address[Ambily]{Statistics and Mathematics Unit, Indian Statistical Institute, \\
 Bangalore 560 059, India}
\address[Rao]{School of Mathematics,Tata Institute of Fundamental Research,\\
Homi Bhabha Road, Mumbai 400 005, India}

\begin{abstract}
We prove that a quadratic $A[T]$-module $Q$ with Witt
index  ($Q/TQ$)$ \geq d$, where $d$ is the dimension of the equicharacteristic  
regular local ring $A$, is extended from $A$. This improves a theorem 
of the second named author who showed it when $A$ is the local ring at a smooth point 
of an  affine variety over an infinite field. To establish our result,  
we need to  establish a Local-Global Principle (of Quillen) for the 
Dickson--Siegel--Eichler--Roy (DSER) elementary orthogonal transformations.

\medskip


\end{abstract}
\end{frontmatter}

\section{Introduction}

Let $A$ be a commutative Noetherian ring in which $2$ is invertible
and let $B$ be the polynomial $A$-algebra $A[X_1, \ldots, X_n]$ in $n$
indeterminates. Let $Q=(Q,q)$ be a quadratic space over $B$ and let
$Q_0=(Q_0,q_0)$ be the  reduction of $Q$ modulo the ideal of $B$
generated by $X_1, \ldots, X_n$.  In \cite{MR0469914}, A. A. Suslin and V. I. Kopeiko 
proved that  if $Q$ is stably extended from $A$ and
for every maximal ideal $\mathfrak{m}$ of $A$, the Witt index of
$A_{\mathfrak{m}}\otimes_A(Q_0, q_0)$  is larger than the Krull dimension
of  $A$, then $(Q,q)$ is extended from $A$. A shorter proof of this, 
due to Inta  Bertuccioni, can be found in \cite{MR0674527} and another
proof is in  the thesis of the second named author. 

In the thesis of the second named author (see \cite{MR727375}, 
\cite{MR748229}), it was  shown that one can improve this result to
Witt index $\geq d$, when $A$ is a local ring at a non-singular point
of an affine variety of dimension $d$  over an infinite
field. Moreover, a question was posed at the end of the thesis 
whether extendability can  be shown for quadratic spaces with Witt 
index $\geq d$ over polynomial extensions of 
any equicharacteristic regular local ring of dimension $d$. 

In this article, we establish this question affirmatively. 

A few words on the proof: The analysis of the equicharacteristic
regular  local ring is done by a patching argument, akin to the one
developed by  Amit Roy in  his paper \cite{MR0656854}. This argument
reduces the problem to the  case of a complete equicharacteristic
regular ring; which is a power series  ring over a field, provided one
can patch the information. 

We found it useful to use Amit Roy's elementary orthogonal 
transformations in \cite{MR0231844} for quadratic spaces with a hyperbolic 
summand over a commutative ring. These transformations (over fields) are 
known as {\it Siegel} transformations or {\it Eichler} transformations in 
the literature: we give a brief historical statement of the development.

These transformations (in matrix form) of quadratic spaces $(V, q)$ over finite 
fields first appeared on pg.12 in L. E. Dickson's book  
``Linear groups: With an exposition of the Galois field theory'' (1958), which 
is an unaltered republication of the first edition (Teubner, Leipzig, 1901). 
Later in ``Sur les groupes 
classiques'' (1948), J. Dieudonn\'{e} extended these results over infinite fields.

These orthogonal transformations (in a matrix form) over 
general fields also appeared in a paper of C. L. Siegel: 
\"{U}ber die analytische Theorie der quadratischen Formen II. Annals of Math. 
36 (1935), 230-263.

Another interpretation occurs in his work \textquotedblleft \"{U}ber die 
Zetafunktionen indefiniter quadratischer Formen\textquotedblright, II.,  Math. Zeitschrift, 
1938, 398-426 (on page 408). Here he used it to define the mass of 
representation of $0$ by an indefinite quadratic form.

M. Eichler studied these transformations of $Q \perp H(k)$ 
in his study of the orthogonal group over fields $k$ and made the first 
systematic use of them in his famous 
book \textquotedblleft Quadratische Formen und orthogonale Gruppen\textquotedblright, 
first published in 1952, and reprinted in 1974. 
 
(Eichler credits Siegel's 1935 
paper for introducing these transformations in the notes on $\S 3$ on pg. 212 
of his book, and also refers to the 1938 Zeitschrift paper of Siegel on 
pg. 218. He does not seem to be aware of Dickson's work.)

Amit Roy studied C. T. C. Wall's paper \cite{MR0155798}, who relied on Eichler's 
book. Amit Roy rewrote the transformations of Eichler in Wall's paper. He 
then generalized these transformations in his thesis (1967) over any 
commutative ring $R$. We shall call these the DSER elementary orthogonal 
transformations or just (Roy's) elementary orthogonal transformation group. 

Note that these transformations of Roy have been further extended to 
form rings by L. N. Vaserstein (when the ring is local), and 
A. Bak (to the general case) in their thesis (respectively). 

We show that  the patching process is possible by establishing a Local-Global 
Principle  for the Elementary Orthogonal group of a quadratic space with a hyperbolic summand. 
For this, we follow the broad outline of A. A. Suslin's method in \cite{MR0472792}  
which led to a $K_1$ analogue of D. Quillen's Local-Global Principle in
\cite{MR0427303}. Instead of using Suslin's `theory of generic forms which 
are elementary', we follow the more `hands on' approach via the yoga of 
commutators. For this, we have to first find an appropriate generating set 
for Roy's group; which is the primary objective of $\S 2$ (That this set 
generates the group is proved in $\S 3$, via V. Suresh's lemma 
in \cite{MR1293626}). We record the commutator calculus in $\S 4$, 
and refer the reader to rigorous proof of these identities to an article 
we have placed in the arXiv \cite{aa}. These commutator calculations 
enable us to prove the Local-Global Principle for Roy's group of orthogonal 
transformations over a polynomial extension.

As an interesting by-product, one realizes from the yoga of commutators
in this elementary orthogonal group that it mimics Tang's well-known
group in some features defined in \cite{MR1609905}, and the unitary
group  of Bass defined in \cite{MR0371994}. The first named author 
intends to pursue the study of  this group in more detail in a 
sequel article, where she hopes to establish  A. Bak's type 
(see \cite{MR1115826}) solvability theorem for the quotient 
group by the elementary subgroup. 

Note: To make the reading effortless, we have placed the onerous (but
straightforward) computations in this group as an article on the arXiv
(see \cite{aa}) which can be accessed by any reader. 

Finally, we have not attempted to study the `$\Lambda$-ring variant' of 
this problem via the variant elementary orthogonal group as defined  by 
A. Bak in his thesis  (see \cite{MR0252431}). We feel that it will throw 
more light on the interrelationship between  all these groups; which will 
be carried out in a separate venture by the first named author.

\section{Preliminaries}

Let $A$ be a commutative ring in which $2$ is invertible. A {\em
quadratic $A$-module} is a pair $(M,q)$, where $M$ is an $A$-module
and $q$ is a quadratic form on $M$. A {\em quadratic space} over $A$
is a pair $(M,q)$, where $M$ is a finitely generated projective
$A$-module and $q:M\longrightarrow A$ is a non-singular quadratic
form. Let $M^*$ denotes the dual of the module $M$. Let $B_q$ be the
symmetric bilinear form associated to $q$ on $M$, which is given by
$B_q(x,y) = q(x+y)-q(x)-q(y)$ and $d_{B_q} : M \rightarrow M^*$ be
the induced isomorphism given by $d_{B_q}(x)(y) = B_q(x,y)$, where
$x,y\in M$. Given two quadratic $A$-modules $(M_1,q_1)$ and
$(M_2,q_2)$, their orthogonal sum $(M,q)$ is defined by taking
$M=M_1\oplus M_2$ and $q((x_1,x_2))=q_1(x_1)+q_2(x_2)$ for $x_1 \in M_1,
x_2 \in M_2$. Denote $(M,q)$ by $(M_1,q_1)\perp(M_2,q_2)$ and $q$
by $q_1 \perp q_2$.

Let $P$ be a finitely generated projective $A$-module. The module $P
\oplus P^*$ has a natural quadratic form given by $p((x,f)) = f(x)$ for
$x\in P$, $f\in P^*$. The corresponding bilinear form $B_p$ is given
by $B_p((x_1,f_1),(x_2,f_2)) = f_1(x_2)+f_2(x_1)$ for $x_1,x_2 \in P$ and
$f_1,f_2 \in P^*$. The quadratic space $(P \oplus P^*, p)$, denoted by
$H(P)$, is called the {\em hyperbolic space} of $P$. A quadratic space
$M$ is said to be hyperbolic if it is isometric to $H(P)$ for some
$P$. The quadratic space $H(A)$, denoted by $h$, is called a
\emph{hyperbolic plane}. The orthogonal sum $h\perp h \perp \cdots \perp h$
of $n$ hyperbolic planes is denoted by $h^n$. A quadratic space $M$ is
said to have {\em Witt index} $\ge n$ if $M\simeq M_0 \perp H(P)$,
where rank $P\ge n$. A quadratic space $M$ is said to have {\em
hyperbolic rank} $\ge n$ if $M \simeq M_0 \perp h^d$, where $d\ge n$. A
quadratic space $M$ is said to be {\em cancellative} if for any
quadratic $A$-spaces $M_1,M_2$ with $M\perp M_2 \simeq M_1\perp M_2$,
then $M \simeq M_1$.

Let $Q$ be a quadratic $A$-space and $P$ be a finitely generated
projective $A$-module. Now let $M = Q \perp H(P)$. This is a
quadratic space with the quadratic form $q\perp p$. The associated
bilinear form on $M$, denoted by $\langle \cdot , \cdot\rangle$, is
given by
\begin{align*} \langle (a,x),(b,y) \rangle &= B_q((a,b)) +
B_p((x,y))\textnormal{ for all } a,b \in Q \textnormal{ and } \,x,y
\in H(P),
\end{align*} where $B_q$ and $B_p$ are the bilinear forms on $Q$ and
$P$.  Let $M = M(B,q)$ be a quadratic module over $A$ with quadratic
form $q$ and associated symmetric bilinear form $B$. Then the
orthogonal group of $M$ is defined as follows:
\begin{equation}\label{oam} \O_A(M) = \{\sigma \in \Aut(M)\mid
q(\sigma(x))=q(x) \textnormal{ for all } x \in M\},
\end{equation} where $\Aut(M)$ be the group of all $A$-linear
automorphisms of $M$.

Let $M$ be a free module of finite rank. By choosing a basis for $M$,
we can define $$\SO_A(M)=\SL(M)\cap \O_A(M),$$ where $\SL(M)$ is the
subgroup of $\Aut(M)$ consists of automorphisms of determinant 1.
This is a normal subgroup of $\O_A(M)$ and is called the {\em special
orthogonal group} of $M$. See \cite{MR1007302} for more details.

For any $A$-linear map $\alpha : Q \rightarrow P$($\beta :
Q\rightarrow P^*$), the dual map $\alpha^t : P^* \rightarrow Q^*$
($\beta^t : P^{**}\simeq P \rightarrow Q^*$) is defined as
$\alpha^t(\varphi) = \varphi \circ \alpha$ ($\beta^t(\varphi^*) =
\varphi^* \circ \beta$) for $\varphi \in P^*$ ($\varphi^* \in P^{**}$). Recall from
\cite{MR0231844}, the $A$-linear map $\alpha^* : P^*\rightarrow Q$
($\beta^* : P\rightarrow Q$) is defined by $\alpha^* = d_{B_q}^{-1}\circ
\alpha^t$ ($\beta^* = d_{B_q}^{-1}\circ \beta^t \circ \varepsilon$,
where $\varepsilon: P\rightarrow P^{**}$ is the natural isomorphism)
and is characterized by the relation
\[(f\circ\alpha)(z) = B_q\left(\alpha^*(f),z \right)  \textnormal{ for } \,f\in
P^*, z\in Q.\]

In \cite{MR0231844}, A. Roy defined the ``elementary'' transformations
$E_{\alpha}, E_{\beta}^*$ of $Q\perp H(P)$ given by
\[\begin{array}{lllll}
\medskip 
E_{\alpha}(z) &= z+\alpha(z) & & E_{\beta}^*(z) &=z+\beta(z)\\ 
\medskip
E_{\alpha}(x) &= x & & E_{\beta}^*(x) &= -\beta^*(x)+ x-\frac{1} {2}\beta\beta^*(x)\\ 
\medskip
E_{\alpha}(f) &= -\alpha^*(f)-\frac{1}{2}\alpha\alpha^*(f)+f & & E_{\beta}^*(f) &= f
\end{array}\] 
for $z \in Q, x \in P$ and $f\in P^*$.  Observe that
these transformations are orthogonal with respect to the above
quadratic form $q \perp p.$

Now we recall the notion of {\em generalized dimension function} from
\cite{MR722004}. Let $\mathcal{P} \subset \spec A$ be a set of primes,
$\mathbb{N}$ be the set of natural numbers and $d: \mathcal{P}
\rightarrow \mathbb{N} \cup \{0\}$ be a function. For primes
$\mathfrak{p}, \mathfrak{q}$ of $\mathcal{P}$, define a partial order
$\ll$ on $\mathcal{P}$ as $\mathfrak{p} \ll \mathfrak{q}$ iff
$\mathfrak{p} \subset \mathfrak{q}$ and $d(\mathfrak{p}) >
d(\mathfrak{q})$. A function $d: \mathcal{P} \rightarrow \mathbb{N}
\cup \{0\}$ is a generalized dimension function if for any ideal $I$
of $A$, $V(I) \cap \mathcal{P}$ has only a finite number of minimal
elements with respect to the partial ordering $\ll$.

We found it difficult to give a meaningful set of commutator relations
for the set of generators 
$\{E_{\alpha}, E^*_{\beta}~|~\alpha \in \hom_A(Q, P), \beta \in \hom_A(Q,P^*)\}$.

Let $Q$ and $P$ be free $A$-modules. In this case, we could conceive of a
natural set of generators, for which we could develop the commutator
machinery. These generators will be denoted by $E_{\alpha_{ij}}$,
$E^*_{\beta_{ij}}$ below. We proceed to define these now. 

\begin{notation} Let $P$ and $Q$ be free modules of rank $m$ and $n$
respectively, then we can identify $P$, $P^*$ and $Q$ with $A^m$, $A^m$
and $A^n$ respectively. Let $\{z_i : 1\leq i\leq n \}$ be a basis for
$Q$, $\{g_i : 1\leq i\leq n \}$ be a basis for $Q^*$, $\{x_i : 1\leq
i\leq m \}$ be a basis for $P$ and $\{f_i : 1\leq i\leq m \}$ be a basis
for $P^*$.

Let $p_i: A^n \longrightarrow A$ be the projection onto the $i^{th}$
component and $\eta_i: A \longrightarrow A^n$ be the inclusion into
the $i^{th}$ component.  Let $\alpha \in \hom(Q,P)$. Let $\alpha_{i},
\alpha_{ij} \in \hom(Q,P)$ be the maps given by 
$$\alpha_{i} = \eta_i\circ p_i\circ\alpha \quad\textnormal{and}\quad \alpha_{ij} =
\eta_i\circ p_i\circ\alpha\circ \eta_j\circ p_j$$ 
for $1 \leq i \leq m$ and $1\leq j \leq n$. 
Clearly $\alpha = {\Sigma_{i=1}^m\alpha_{i}}
= {\Sigma_{i=1}^m\Sigma_{j=1}^n\alpha_{ij}}.$ Then 
$\alpha_{i}^*,\alpha_{ij}^* \in \hom(P^*,Q)$ is the maps given by 
$$\alpha_{i}^* = {(\alpha^*)}_i = \alpha^*\circ \eta_i\circ p_i \quad \textnormal{and}\quad \alpha_{ij}^* = (\alpha^*)_{ij}= \eta_j\circ p_j\circ\alpha^*\circ \eta_i\circ p_i.$$  Then $\alpha^* = {\Sigma_{i=1}^m\alpha_{i}^*} = {\Sigma_{i=1}^m\Sigma_{j=1}^n\alpha_{ij}^*}$. One can also see that this definition of $\alpha_{i}^*,\alpha_{ij}^*$ coincides with the one obtained by applying $\alpha^* = {d_{B_q}}^{-1}\circ \alpha^t \in \hom(P^*,Q)$ to $\alpha_{i}$ and $\alpha_{ij}$.

\vspace{2mm} Let $z = {\Sigma_{j=1}^n d_jz_j}\in Q$ for $d_j \in
A\,(1\le j \le n)$.  Then $\alpha$ is given by  $\alpha(z_j) = x^{(j)} =
{\Sigma_{i=1}^m b_{ij}x_i} $  for $b_{ij} \in A\,(1\le i\le m)$ and  $\alpha(z)
= {{\Sigma_{j=1}^n\Sigma_{i=1}^m d_jb_{ij} x_i}}$, $\alpha_i(z) =
{{\Sigma_{j=1}^n d_jb_{ij} x_i}}$ and  $\alpha_{ij}(z) = d_jb_{ij}
x_i$.  Let $\alpha^*(f_i) = w_i$ for some $w_i \in Q$. If $f = {\Sigma_{i=1}^m c_if_i}$ for $c_i \in A\,(1\le i\le m)$, then $c_i = \langle
f,x_i \rangle$ and so $\alpha^*(f) = {\Sigma_{i=1}^m \langle f,x_i
\rangle w_i}$. If $w_i = {{\Sigma_{j=1}^n y_jz_j}}$ for some $y_j \in
A$, then $w_{ij} = y_jz_j\in Q$. 

\vspace{2mm}
For $1 \leq i \leq m$ and $1 \leq j \leq n$, the maps $\alpha_i^*$ and
$\alpha_{ij}^*$'s are given by
\[ \begin{array}{llll} 
 \alpha_i^*(f_j) = \begin{cases} 
		      w_i &\textnormal{if}\quad  j=i,\\ 
		      0 &\textnormal{if}\quad    j \neq i.
		   \end{cases} & & &
\alpha_{ij}^*(f_k) = \begin{cases}
			 w_{ij}&\textnormal{if}\quad k=i,\\ 
			 0 &\textnormal{if}\quad k \neq i.
		      \end{cases}
      \end{array} \]
Let $\beta \in \hom(Q,P^*)$. Set $\beta^*(x_i) = v_i$ for
some $v_i \in Q$, let $v_{ij}$ denotes the element $\eta_j\circ
p_j(v_i) $. Now defining the maps $\beta_i, \beta_{ij}, \beta_i^*$,
$\beta_{ij}^*$ similarly and extending these to the whole of $Q\oplus
P\oplus P^{\ast}$, we get the maps as follows:  For $z\in Q$,
$x\in P$, $f\in P^*$; $1\leq i\leq m$ and $1\leq j\leq n$;
\[\begin{array}{llllll} 
\medskip
\alpha_{ij}(z,x,f) &= \left( 0,\langle w_{ij},z \rangle x_i,0 \right),  & &\beta_{ij}(z,x,f) &=
\left(0,0,\langle v_{ij},z \rangle f_i\right),\\
\medskip
\alpha_i(z,x,f) &= \left( 0,\langle w_i,z \rangle x_i,0 \right), &
&\beta_i(z,x,f) &= \left(0,0,\langle v_i,z \rangle f_i\right),\\
\medskip
\alpha(z,x,f) &= \left( 0,{\Sigma_{i=1}^m \langle w_i,z \rangle
x_i},0\right),& &\beta(z,x,f) &= \left(0,0,{\Sigma_{i=1}^m \langle
v_i,z \rangle f_i}\right),\\
\medskip
\alpha_{ij}^*(z,x,f) &= \left(\langle f,x_i \rangle w_{ij},0,0
\right),& &\beta^*_{ij}(z,x,f) &= \left(\langle x,f_i \rangle
v_{ij},0,0 \right),\\
\medskip
\alpha_i^*(z,x,f) &= \left(\langle f,x_i \rangle w_i,0,0 \right),&
&\beta^*_i(z,x,f) &= \left(\langle x,f_i \rangle v_i,0,0 \right),\\
\medskip
\alpha^*(z,x,f) &= \left({\Sigma_{i=1}^m \langle f,x_i \rangle
w_i},0,0 \right), & &\beta^*(z,x,f) &= \left({\Sigma_{i=1}^m \langle
x,f_i \rangle v_i},0,0 \right).
\end{array}\] 
Also, $q(w_{ij}) = \frac{1}{2}\langle w_{ij},w_{ij}
\rangle$ and $q(v_{ij}) = \frac{1}{2}\langle v_{ij},v_{ij} \rangle$.

\vspace{3mm}
For $\alpha \in \hom(Q,P)$, the orthogonal transformation $E_{{\alpha}_{ij}}$ on $Q\perp H(P)$  is
given by
\begin{align*} E_{\alpha_{ij}}(z,x,f) &=
\left(I-\alpha^*_{ij}+\alpha_{ij}-\frac{1}{2}\alpha_{ij}\alpha^*_{ij}\right)(z,x,f)\\
&= \left(z-\langle f,x_i \rangle w_{ij},\;x+\langle w_{ij},z\rangle
x_i -\langle f,x_i \rangle q(w_{ij})x_i,\;f\right).            
\end{align*}

For $\beta \in \hom(Q,P^*)$, the orthogonal transformation $E_{{\beta}_{ij}}^*$ of $Q\perp H(P)$ is
given by
\begin{align*} E_{\beta_{ij}}^*(z,x,f) &=
\left(I-\beta_{ij}^*+\beta_{ij}-\frac{1}{2}\beta_{ij}\beta_{ij}^*\right)(z,x,f)\\
&= \left(z-\langle f_i,x \rangle v_{ij},\;x,\;f+\langle
v_{ij},z\rangle f_i -\langle x,f_i \rangle q(v_{ij})f_i\;\right).            
\end{align*}
\end{notation}
\section{Roy's Elementary orthogonal transformations} In this section,
we consider the orthogonal group of $Q\perp H(P)$, 
denoted by $\O_A(Q\perp H(P))$, where $Q$ and $P$ are free $A$-modules
of finite rank. Precisely,
\[\O_A(Q\perp H(P)) = \{\sigma \in \Aut(Q\perp H(P))\mid (q\perp
p)(\sigma(z,y))= (q \perp p)(z,y)\; \forall\; (z,y) \in Q\perp
H(P)\}.\] Since $Q$ and $P$ are free modules, the elements of
$\O_A(Q\perp H(P))$ can be represented as matrices over $A$ by
choosing a basis for $Q$ and $P$. Then we can identify $\O_A(Q\perp
H(P))$ as a subgroup of $\GL_{(n+2m)}(A)$. 
\begin{lemma}  An $(n+2m)\times(n+2m)$ matrix $T = \begin{pmatrix}
						    A&B&C\\ D&F&G\\ H&J&K
						  \end{pmatrix}$ belongs to
$\O_A(Q\perp H(P))$  if and only if any of the following equations
hold.
\begin{enumerate}[label= \em(\alph{*})]
 \item $T^t \psi T = \psi,$ for $\psi = \begin{pmatrix} \phi&0&0\\
0&0&I\\ 0&I&0
                                    \end{pmatrix} $, where $\phi$ is
the matrix corresponding to the non-singular quadratic form $q$ on $Q$ and
$\begin{pmatrix} 0&I\\ I&0
\end{pmatrix}$
is the matrix of the hyperbolic form $p$.
\vspace{2mm}
\item 
    $ \begin{pmatrix} 
	      \phi^{-1} A^t\phi&\phi^{-1}H^t&\phi^{-1}D^t\\ 
	      C^t\phi&K^t&G^t\\ 
	      B^t\phi&J^t&F^t
	\end{pmatrix} . \begin{pmatrix}
				A&B&C\\ D&F&G\\ H&J&K
			 \end{pmatrix} = Id$.
\end{enumerate}
\end{lemma}
\begin{proof} 
Follows immediately from the definition of $\O_A(Q\perp H(P))$.
\end{proof}   
Let $\EO_A(Q\perp H(P))$ be the subgroup of
$\O_A(Q\perp H(P))$ generated by $E_\alpha$ and $E_\beta^*$, where
$\alpha \in \hom(Q,P)$ and $\beta \in \hom(Q,P^*)$. We call this group
{\em elementary orthogonal group} and these transformations {\em
elementary orthogonal transformations}. If $Q$ and $P$ are free modules of rank $n$ and $m$ respectively,  we have the elementary transformations of the type $E_{\alpha_{ij}}$ and $E_{\beta_{ij}}^*$ for $1 \leq i \leq m$, $1 \leq j \leq n$. 

\vspace{1mm}

Next, we compare the elementary orthogonal group of Roy's
elementary transformations and that of the Dickson-Siegel-Eichler
transformations which is defined as follows: 

\begin{definition}\cite[{Chapter~5}]{MR1007302} {\em Let $(M,B,q)$ be a non-degenerate quadratic module over $A$ and let $\O_A(M)$ be its orthogonal group. Let $u$ and $v$
be in $M$ with $u$ isotropic and $B(u,v)=0$. For $r=q(v)$, define the
Dickson-Siegel-Eichler transformation $\Sigma_{u,v,r} \in \edm(M)$,
by}
\begin{equation*} \Sigma_{u,v,r}(x) = x+uB(v,x)-vB(u,x)-urB(u,x).
\end{equation*}
\end{definition} One can easily verify the following properties of
Eichler transformations.
\begin{enumerate}[label= (\roman{*})]
	  \item $\Sigma_{u,v,q(v)} \in \O_A(M)$,
	  \item $\Sigma_{u,v,q(v)} \Sigma_{u,w,q(w)} =
		  \Sigma_{u,v+w,q(v)+q(w)+h(v,w)}$,
	  \item $\Sigma_{u,v,q(v)}^{-1} = \Sigma_{u,-v,q(v)}$,
	  \item $\sigma\Sigma_{u,v,q(v)}\sigma^{-1} = \Sigma_{\sigma u,\sigma
		  v,q(v)}$ for $\sigma \in \O_A(M)$.
\end{enumerate}

Observe that $\Sigma_{0,0,0}=Id$. 

\vspace{1mm}

We may regard the elementary orthogonal transformations
$E_{\alpha_{ij}}$ and $E_{\beta_{ij}}^*$ as Dickson-Siegel-Eichler
transformations. More precisely, the orthogonal transformation $E_{{\alpha}_{ij}}$ of $M=Q\perp H(P)$
given by
 \begin{equation*} E_{\alpha_{ij}}(z,x,f) = (z-\langle f,x_i \rangle
w_{ij},\;x+\langle w_{ij},z\rangle x_i -\langle f,x_i \rangle
q(w_{ij})x_i,\;f)
 \end{equation*} can be written as
$\Sigma_{x_i,w_{ij},q(w_{ij})}(z,x,f)$. For,
\begin{align*}
\Sigma_{x_i,w_{ij},q(w_{ij})}(z,x,f)=&(z,x,f)+(0,x_i,0)\langle
(w_{ij},0,0),(z,x,f)\rangle-(w_{ij},0,0)\\&\langle
(0,x_i,0),(z,x,f)\rangle-(0,x_i,0)q(w_{ij})\langle
(0,x_i,0),(z,x,f)\rangle\\ =&(z-\langle f,x_i \rangle
w_{ij},\;x+\langle w_{ij},z\rangle x_i -\langle f,x_i \rangle
q(w_{ij})x_i,\;f).
\end{align*} Similarly, the orthogonal transformation
$E_{{\beta}_{ij}}^*$ of $M$ given by
\begin{equation*} E_{\beta_{ij}}^*(z,x,f)= (z-\langle f_i,x \rangle
v_{ij},\;x,\;f+\langle v_{ij},z\rangle f_i -\langle x,f_i \rangle
q(v_{ij})f_i\;)
\end{equation*} can be written as
$\Sigma_{f_i,v_{ij},q(v_{ij})}(z,x,f)$.

These elementary orthogonal transformations also satisfy the
properties listed above. Moreover, as we saw in the previous section,
they satisfy a more general set of properties analogous to
Property(ii). 

The transformations defined by A. Roy \cite{MR0231844} can also be
viewed as {\em unitary transvections} \cite[{Section~5}]{MR0371994} of
certain types of quadratic modules over a {\em unitary ring}
$(A,\lambda, \Lambda)$. See \cite[{Section~4}]{MR0371994} for further details of unitary rings.
%The notion of unitary ring is due to A. Bak \cite{Bakthesis}.

Let $M = V \perp H(P)$. If $x = (v;p,q) \in M$, we have $f(x,x) =
f(v,v) +{\langle q, p \rangle}_P$. Suppose $P$ has a unimodular
element $p_0$. i.e. there is a $q_0 \in \overline{P}$ such that
${\langle q_0, p_0 \rangle}_P = 1$. For any elements $p_0 \in P, w_0
\in V$ and $a_0 \in A$ with $a_0\equiv f(w_0,w_0)\mod \Lambda$,
\begin{align*} f(p_0,p_0) &\in \Lambda,\\ \langle w_0, p_0 \rangle &=
0,\\ f(w_0,w_0)& \equiv a_0 \mod \Lambda.
\end{align*}  If $x = (v;p,q)$, then 
\[\sigma_{p_0,a_0,w_0}(x)  = x +
p_0 \langle w_0, x \rangle  - w_0  \overline{\lambda} \langle p_0, x
\rangle  - p_0 \overline{\lambda} a_0 \langle p_0, x \rangle .\] Now
take $\Lambda = 0, \lambda =1, f(w_0,w_0) = a_0$ and $\langle w_0, w_0
\rangle = 2 f(w_0,w_0) = 2a_0$. Then we get 
\begin{align*} E_{\alpha_{ij}}(z,x,f) &= \sigma_{x_i,\frac{\langle
w_{ij}, w_{ij} \rangle}{2}, w_{ij}}(z,x,f), \\ E_{\beta_{ij}}^*(z,f,x)
&= \sigma_{f_i,\frac{\langle v_{ij}, v_{ij} \rangle}{2},
v_{ij}}(z,f,x). 
\end{align*}

Now we state the splitting property and extend the Lemma~1.4 of 
\cite{MR1293626} regarding Roy's transformations. 
We use the notation $E(\alpha)$ for either $E_{\alpha}$ or $E_{\alpha}^*$, where $\alpha
\in \hom(Q,P)$ or $\hom(Q,P^*)$ respectively. Combining
Lemma~1.2 and Lemma~1.3 of \cite{MR1293626}, we have the following:
\begin{lemma}[{Splitting property \cite{MR1293626}}] For {\em
$\alpha_1,\alpha_2 \in \hom(Q,P)$} or { $\hom(Q,P^*)$ } we have
\[ E\left(\alpha_1+\alpha_2\right) =
E\left(\frac{\alpha_1}{2}\right)E\left(\alpha_2\right)E\left(\frac{\alpha_1}{2}\right)
=
E\left(\frac{\alpha_2}{2}\right)E\left(\alpha_1\right)E\left(\frac{\alpha_2}{2}\right).
\]
 \end{lemma} 
The following lemma extends Lemma~1.4 of \cite{MR1293626}.
\begin{lemma} With the notation as above, the group $\EO_A\left(Q\perp
H(P)\right)$ is generated by $E({\alpha_{ij}})$, with {\em $\alpha \in
\hom(Q,P)$} or {\em $\hom\left(Q,P^{\ast}\right)$}; $1\leq i \leq m$
and $1\leq j \leq n$.
\end{lemma}
\begin{proof} For $\alpha \in \hom(Q,P)$ or $\hom(Q,P^{\ast})$, we
have $\alpha = {\sum_{i=1}^m\sum_{j=1}^n\alpha_{ij}}$ from the
previous section. By repeated applications of the splitting property, we
have
\begin{align*} E\left(\alpha\right)=
&E\left(\frac{\alpha_{11}}{2}\right)E\left(\frac{\alpha_{21}}{2}\right)\cdots
E\left(\frac{\alpha_{m1}}{2}\right)
E\left(\frac{\alpha_{12}}{2}\right)\cdots
E\left(\frac{\alpha_{m2}}{2}\right)\\ &\cdots
E\left(\frac{\alpha_{(m-1)n}}{2}\right)E\left(\alpha_{mn}\right)E\left(\frac{\alpha_{(m-1)n}}{2}\right)\cdots
E\left(\frac{\alpha_{11}}{2}\right).
\end{align*} This proves the lemma.
\end{proof}
\section{Commutator relations between elementary generators} All the main results in this paper depend on various commutator relations between the generators of $\EO_A(Q\perp H(P))$. The computations for these relations are messy and the general expressions and their detailed proofs are given in a note posted in the arXiv at \cite{aa}. In this section, we state some of the commutator relations which will be used in the sections to follow.
\begin{lemma}\label{l21} Let $Q,P$ be free $A$-modules of rank $n$ and
$m$ respectively$;$ {\em $\alpha,\delta \in  \hom(Q,P)$} and {\em
$\beta,\gamma \in \hom(Q,P^*)$}. Then for any given $i,j,k,l$ with $i \neq k$ for $1\leq i,k \leq m$ and $1 \leq j,l \leq n;$ we have the following commutator relations between the elementary orthogonal
transformations $E_{\alpha_{ij}}, E_{\delta_{kl}}, E_{\beta_{kl}}^{\ast} $ and $E_{\gamma_{kl}}^*:$
\begin{enumerate}[label= \em(\roman{*})]
 \item $[E_{\alpha_{ij}}, E_{\delta_{kl}}] = I +
\delta_{kl}\alpha^{\ast}_{ij} - \alpha_{ij} \delta^{\ast}_{kl}$,
\item $[E_{\alpha_{ij}},E_{\beta_{kl}}^{\ast}] = I-\alpha_{ij}
\beta^{\ast}_{kl}+ \beta_{kl}\alpha^{\ast}_{ij}$,
\item $[E_{\beta_{ij}}^*, E_{\gamma_{kl}}^*] = I + \gamma_{kl}
\beta^{\ast}_{ij} - \beta_{ij} \gamma^{\ast}_{kl}$ .
\end{enumerate}
\end{lemma}
\begin{proof}
(i), (ii) and (iii) follows from  Lemma~2.1, Lemma~2.3 and Lemma~2.7 of \cite{aa} respectively.
\end{proof}
\begin{corollary}\label{c07} Under the same assumptions as in Lemma~\ref{l21} and for $a,b,c,d \in A$, we have the following$:$
\begin{enumerate}[label= \em(\roman{*})]
\item $\left[E_{a\alpha_{ij}},
E_{b\delta_{kl}}\right]= \left[E_{c\alpha_{ij}}, E_{d\delta_{kl}} \right]$ if
$ab=cd$, 
\item
$\left[E_{a\alpha_{ij}},E_{b\beta_{kl}}^{\ast}\right]=\left[E_{c\alpha_{ij}},E_{d\beta_{kl}}^{\ast}\right]$
if $ab=cd$, 
\item $\left[E_{a\beta_{ij}}^*,
E_{b\gamma_{kl}}^*\right]=\left[E_{c\beta_{ij}}^*,
E_{d\gamma_{kl}}^*\right]$ if $ab=cd$.
\end{enumerate}
\end{corollary} 
\begin{proof}
(i), (ii) and (iii) follows from Corollary~2.2, Corollary~2.5 and Corollary~2.8 of \cite{aa} respectively.
\end{proof}
\begin{lemma}\label{l22} Let $Q,P$ be free $A$-modules of rank $n$ and
$m$ respectively$;$ {\em $\alpha,\delta \in  \hom(Q,P)$} and {\em
$\beta,\gamma \in \hom(Q,P^*)$}. For any given $i,j,k,l,p,q$ with
$i \neq k$, $k \neq p$ for $1 \leq i,k,p \leq m$ and $1 \leq j,l,q \leq
n;$  we have the following commutator relations$:$
\begin{enumerate}[label= \em(\roman{*})]
\item $\left[E_{\beta_{ij}}^*,\left[E_{\alpha_{kl}},E_{\delta_{pq}}\right]\right] =
E_{\lambda_{kj}}\left[E_{\beta_{ij}}^*,E_\frac{\lambda_{kj}}{2}\right]$,
\item $\left[E_{\alpha_{ij}},\left[E_{\delta_{kl}},E_{\beta_{pq}}^*\right]\right] =
E_{\mu_{kj}}\left[E_{\alpha_{ij}},E_\frac{\mu_{kj}}{2}\right]$,
\item $\left[E_{\beta_{ij}}^*,\left[E_{\gamma_{kl}}^*,E_{\alpha_{pq}}\right]\right] =
E_{\nu_{kj}}^*\left[ E_{\beta_{ij}}^*,E_\frac{\nu_{kj}}{2}^* \right]$, 
\item $\left[E_{\alpha_{ij}},\left[E_{\beta_{kl}}^*,E_{\gamma_{pq}}^*\right]\right] =
E_{\xi_{kj}}^*\left[E_{\alpha_{ij}},E_\frac{\xi_{kj}}{2}^*\right]$, 
\end{enumerate}
where \[\lambda_{kj}=\alpha_{kl}\delta_{pq}^*\beta_{ij},\;\mu_{kj}=\delta_{kl}\beta_{pq}^*\alpha_{ij}
\in \hom(Q,P), \]
\[\nu_{kj}=\gamma_{kl}\alpha_{pq}^*\beta_{ij} \textnormal{ and }
\xi_{kj}=\beta_{kl}\gamma_{pq}^*\alpha_{ij}\in \hom(Q,P^*).\]
\end{lemma}
\begin{proof}
(i), (ii), (iii) and (iv) follows from Lemma~3.1, Lemma~3.3, Lemma~3.5 and Lemma~3.7 of \cite{aa} respectively.
\end{proof}
\begin{corollary}\label{c08} Under the same assumptions as in Lemma~
\ref{l22} and for $a,b,c,d,e,f \in A$, we have the following$:$
\begin{enumerate}[label=\em(\roman{*})]
\item $\left[E_{a\beta_{ij}}^*,\left[E_{b\alpha_{kl}},E_{c\delta_{pq}}\right]\right]$ =
$\left[E_{d\beta_{ij}}^*,\left[E_{e\alpha_{kl}},E_{f\delta_{pq}}\right]\right]$ if $abc=def$
and $a^2bc=d^2ef$, 
\item $\left[E_{a\alpha_{ij}},\left[E_{b\delta_{kl}},E_{c\beta_{pq}}^*\right]\right] =
\left[E_{d\alpha_{ij}},\left[E_{e\delta_{kl}},E_{f\beta_{pq}}^*\right]\right]$ if $abc=def$
and $a^2bc=d^2ef$, 
\item $\left[E_{a\beta_{ij}}^*,\left[E_{b\gamma_{kl}}^*,E_{c\alpha_{pq}}\right]\right]=\left[E_{d\beta_{ij}}^*,\left[E_{e\gamma_{kl}}^*,E_{f\alpha_{pq}}\right]\right]$
if $abc=def$ and $a^2bc=d^2ef$, 
\item $\left[E_{a\alpha_{ij}},\left[E_{b\beta_{kl}}^*,E_{c\gamma_{pq}}^*\right]\right]=\left[E_{d\alpha_{ij}},\left[E_{e\beta_{kl}}^*,E_{f\gamma_{pq}}^*\right]\right]$
if $abc=def$ and $a^2bc=d^2ef$.  
\end{enumerate}
\end{corollary}
\begin{proof}
(i), (ii), (iii) and (iv) follows from Corollary~3.2, Corollary~3.4, Corollary~3.6 and Corollary~3.8 of \cite{aa} respectively.
\end{proof}
\section{Local-Global Principle for Roy's Elementary Orthogonal
Transformations} In this section, we establish that $EO_{A[X]}(M[X])$,
where $M = Q \perp H(P)$ such that $Q$ and $P$ are free modules of
rank $n$ and $m$ respectively, satisfies the Local-Global principle. 
\begin{theorem}[\textbf{Local-Global Principle}]\label{t11}
Let $\theta(X) \in \O_{A[X]}(M[X])$. If for
all maximal ideals $\mathfrak{m}$ of $A$, $\theta(X)_\mathfrak{m} \in
\O_{A_{\mathfrak{m}}}(M_{\mathfrak{m}})\cdot \EO_{A_{\mathfrak{m}}[X]}(M_\mathfrak{m}[X])$ , then $\theta(X) \in
\O_A(M)\cdot \EO_{A[X]}(M[X])$.
\end{theorem} Before beginning the proof we make the following remark.
\begin{remark} Replacing $\theta(X)$ by $\theta(0)^{-1}\theta(X)$, we
may assume that $\theta(0) = 1$. Further for any ring $A$, $\theta(X)
\in \O_A(M)\EO_{A[X]}(M[X])$ and $\theta(0) = Id$ implies that $\theta(X)
\in \EO_{A[X]}(M[X])$.\\ 
For, if $\theta(X) = \gamma \varepsilon(X), \gamma \in
\O_A(M)$ and $\varepsilon(X) \in \EO_{A[X]}(M[X])$, then $\gamma = \theta(0)
\varepsilon(0)^{-1} = \varepsilon(0)^{-1}$. 
\end{remark} In view of this remark we can rewrite the
Theorem~\ref{t11} as follows:
\begin{theorem}[\textbf{Local-Global Principle}]\label{t00}
Let $\theta(X) \in \O_{A[X]}(M[X])$ with $\theta(0) = Id$. If for all maximal ideals $\mathfrak{m}$ of $A$, $\theta(X)_\mathfrak{m} \in EO_{A_{\mathfrak{m}}[X]}(M_\mathfrak{m}[X])$,
 then $\theta(X) \in \EO_{A[X]}(M[X])$.
\end{theorem} 
We begin with some lemmas.
\begin{lemma}\label{l07} Let $G$ be a group and $a_i, b_i\in G$, for
$i=1,...,n.$ Then
\begin{equation*} \textstyle{\prod_{i=1}^n a_ib_i = \prod_{i=1}^n r_ib_ir_i^{-1}
\prod_{i=1}^n a_i,}
\end{equation*} where $r_i = {\prod_{j=1}^i a_j}$.
\end{lemma}
\begin{proof} 
Direct computation.
\end{proof}
\begin{lemma}\label{l08} The group $\EO_{A[X]}(M[X])$ is generated by the
elements of the type  \\$\gamma E\left(X\alpha_{ij}(X)\right)
\gamma^{-1}$, where $\gamma \in \EO_A(M)$, $\alpha_{ij}(X) \in
\hom(Q[X],P[X])$ or $\hom(Q[X],P^{\ast}[X])$.
\end{lemma}
\begin{proof} Let $\theta(X)$ be an element of $\EO_{A[X]}(M[X])$ such that
$\theta(0) = Id$. Then
\begin{align*} 
	  \theta(X)&= \textstyle{\prod_{k=1}^r
	    E\left(\alpha_{i_kj_k}(X)\right) = \prod_{k=1}^r
	    E\left(\alpha_{i_kj_k}(0)+ X{\alpha'}_{i_kj_k}(X)\right)} \\ 
		  &= \textstyle{\prod_{k=1}^r E\left(\frac{\alpha_{i_kj_k}(0)}{2}\right)
			E\left(X{\alpha'}_{i_kj_k}(X)\right)
			E\left(\frac{\alpha_{i_{k}j_{k}}(0)}{2}\right)}
			(\textnormal{ by Splitting property })\\ 
		   &= \textstyle{\prod_{k=1}^{r+1} a_kb_k,}
\end{align*} where
$$\begin{array}{lllll}
\medskip
a_1  &=& E\left(\frac{\alpha_{i_1j_1}(0)}{2}\right), & b_k = E\left(X{\alpha'}_{i_kj_k}(X)\right) &\textnormal{for } k = 1,...,r,\\
\medskip
a_k  &=& E\left(\frac{\alpha_{i_{k-1}j_{k-1}}(0)}{2}\right)E\left(\frac{\alpha_{i_kj_k}(0)}{2}\right)\;&\textnormal{for}\;k = 2,...,r,&   \\ 
a_{r+1}  &=& E\left(\frac{\alpha_{i_rj_r}(0)}{2}\right), & b_{r+1} = 1.& 
\end{array}$$
By Lemma~\ref{l07}, we have
\begin{equation*} 
\theta(X) = \textstyle{\prod_{k=1}^{r+1} \gamma_k
E\left(X{\alpha'}_{i_kj_k}(X)\right) {\gamma_k}^{-1} \prod_{k=1}^{r+1}
a_k,}
\end{equation*} where $\gamma_k = {\prod_{j=1}^{k}
a_j}\in \EO_A(M)$ and ${\prod_{k=1}^{r+1} a_k} =
\prod_{k=1}^r E(\alpha_{i_kj_k}(0)) = \theta(0) = Id.$\\ Therefore
\[ \theta(X) = \textstyle{\prod_{k=1}^{r+1} \gamma_k
E\left(X{\alpha'}_{i_kj_k}(X)\right) {\gamma_k}^{-1}.}
\]
 \end{proof}
\begin{lemma}\label{l09} Let $\alpha, \delta \in \hom(Q,P)$, $\beta,
\gamma \in \hom(Q,P^{\ast})$ and $s$ be a non-nilpotent element of
$A$. Fix $r\in \mathbb{N}$. Let $i,k,p_t \in \{1,2,...,m\}$ and
$j,l,q_t \in \{1,2,...,n\}$ for every $t\in\mathbb{N}$. Then for
sufficiently large $d$, there exists a product decomposition for
$E\left(\frac{a}{s^r}X_{ij}\right)E\left(s^dxY_{kl}\right)E\left(-\frac{a}{s^r}X_{ij}\right)$
in $\EO_{A_s}(M_s)$ given by
\[
E\left(\frac{a}{s^r}X_{ij}\right)E\left(s^dxY_{kl}\right)E\left(-\frac{a}{s^r}X_{ij}\right)
= \prod_{t=1}^\nu E\left(s^{d_t}x_tZ_{p_tq_t}\right),
\] where $X,Y,Z \in \{ \alpha,\beta,\gamma,\delta \}$, $a, x \in A$ and the elements $x_t \in A$ for $t \in \mathbb{N}$ are chosen suitably.
\end{lemma}
\begin{proof} To prove the lemma it is enough to consider the
following cases:

\medskip

\noindent\textbf{Case 1:}
  $(X,Y) \in \{(\alpha,\alpha),(\alpha,\delta),(\beta,\beta),(\beta,\gamma) \}$
  
  \smallskip
 
    $E\left(\frac{a}{s^r}X_{ij}\right)E\left(s^dxY_{kl}\right)E\left(\frac{a}{s^r}X_{ij}\right)^{-1}
= {\prod_{t=1}^\nu E\left(s^{d_t}x_tZ_{p_tq_t}\right)}.$

\medskip
\noindent{{\em Subcase} (a)}: $i\neq k$
\begin{align*}
E\left(\frac{a}{s^r}X_{ij}\right)E\left(s^dxY_{kl}\right)E\left(\frac{a}{s^r}X_{ij}\right)^{-1}
&=
\left[E\left(\frac{a}{s^r}X_{ij}\right),E\left(s^dxY_{kl}\right)\right]E\left(s^dxY_{kl}\right)\\
&=
\left[E\left(as^pX_{ij}\right),E\left(s^qxY_{kl}\right)\right]E\left(s^dxY_{kl}\right)\\
&\quad\quad\textnormal{ (by Corollary ~\ref{c07}~(i)}\;\textnormal{and
Corollary ~\ref{c07}~(iii)} )\\ &= \textstyle{{\prod_{t=1}^\nu
E\left(s^{d_t}x_tZ_{p_tq_t}\right)} \quad \textnormal{for}\; d_t>0.}
\end{align*}
\noindent This equation holds for any positive integers $p,q$ with
$p+q = d-r$.

\medskip

\noindent{ {\em Subcase} (b)}: $i=k$
\begin{align*}
E\left(\frac{a}{s^r}X_{ij}\right)E\left(s^dxY_{kl}\right)E\left(\frac{a}{s^r}X_{ij}\right)^{-1}
&=
\left[E\left(\frac{a}{s^r}X_{ij}\right),E\left(s^dxY_{kl}\right)\right]E\left(s^dxY_{kl}\right)\\
&= E\left(s^dxY_{kl}\right).\\
&\textnormal{ (by Lemma~\ref{l21}(i) and by Lemma~\ref{l21}~(iii)})
\end{align*}
\noindent\textbf{Case 2:} $(X,Y) \in \{(\alpha,\beta),(\beta,\alpha)\}$
\[
E\left(\frac{a}{s^r}X_{ij}\right)E\left(s^dxY_{kl}\right)E\left(\frac{a}{s^r}X_{ij}\right)^{-1}
= \textstyle{\prod_{t=1}^\nu E\left(s^{d_t}x_tZ_{p_tq_t}\right)}.
\]
\noindent{{\em Subcase} (a)}: $i\neq k$

\smallskip

\noindent For instance,
\begin{align*}
E\left(\frac{a}{s^r}\alpha_{ij}\right)E\left(s^dx\beta_{kl}\right)E\left(\frac{a}{s^r}\alpha_{ij}\right)^{-1}
&=
E_{\frac{a}{s^r}\alpha_{ij}}E_{s^dx\beta_{kl}}^*E_{\frac{a}{s^r}\alpha_{ij}}^{-1}\\
&=
\left[E_{\frac{a}{s^r}\alpha_{ij}},E_{s^dx\beta_{kl}}^*\right]E_{s^dx\beta_{kl}}^*\\
&=
\left[E_{as^p\alpha_{ij}},E_{s^qx\beta_{kl}}^*\right]E_{s^dx\beta_{kl}}^*\;\textnormal{ (by Corollary ~\ref{c07}~(ii)})\\ &= \textstyle{\prod_{t=1}^\nu
E(s^{d_t}x_tZ_{p_tq_t})} \;\textnormal{for}\;  d_t>0 \;
\textnormal{and}\; \nu\leq 5.
\end{align*}
\noindent{{\em Subcase} (b)}: $i=k$

\smallskip

\noindent For instance,
\begin{equation}\label{e02}
E\left(\frac{a}{s^r}\alpha_{ij}\right)E\left(s^dx\beta_{il}\right)E\left(\frac{a}{s^r}\alpha_{ij}\right)^{-1}
=
E_{\frac{a}{s^r}\alpha_{ij}}E_{s^dx\beta_{il}}^*E_{\frac{a}{s^r}\alpha_{ij}}^{-1}.
\end{equation}
\noindent Set $d = N_1+N_2+N_3$ such that $N_1\geq r+2$ and $N_2+N_3
\geq 2r+4$.  Now, replacing $E_{s^dx\beta_{il}}^*$ by
$\left[E_{s^{N_1}\alpha_{kl}},\left[E_{s^{N_2}x\beta_{il}^*},E_{s^{N_3}\gamma_{pq}^*}\right]
\right]\left[E_{s^dx\frac{\beta_{il}^*}{2}},E_{s^{N_1}\alpha_{kl}}\right]$
in Equation~\eqref{e02}, using Lemma~\ref{l22}~(i), we have 
\begin{equation*}
E_{\frac{a}{s^r}\alpha_{ij}}E_{s^dx\beta_{il}}^*E_{\frac{a}{s^r}\alpha_{ij}}^{-1}
=
E_{\frac{a}{s^r}\alpha_{ij}}\left[E_{s^{N_1}\alpha_{kl}},\left[E_{s^{N_2}x\beta_{il}^*},E_{s^{N_3}\gamma_{pq}^*}\right]
\right]\left[E_{s^dx\frac{\beta_{il}^*}{2}},E_{s^{N_1}\alpha_{kl}}\right]E_{\frac{a}{s^r}\alpha_{ij}}^{-1}.
\end{equation*}
\noindent Then we will see that the following are in the required product
form.
 \begin{enumerate}[label= (\roman*)] 
\item $E_{\frac{a}{s^r}\alpha_{ij}}E_{s^{N_1}\alpha_{kl}}E_{\frac{a}{s^r}\alpha_{ij}}^{-1}$,
\item
$E_{\frac{a}{s^r}\alpha_{ij}}\left[E_{s^{N_2}x\beta_{il}^*},E_{s^{N_3}\gamma_{pq}^*}\right]E_{\frac{a}{s^r}\alpha_{ij}}^{-1}$,
\item
$E_{\frac{a}{s^r}\alpha_{ij}}\left[E_{s^dx\frac{\beta_{il}^*}{2}},E_{s^{N_1}\alpha_{kl}}\right]E_{\frac{a}{s^r}\alpha_{ij}}^{-1}$.
\end{enumerate}
\begin{align*}
\textnormal{ For, (i) } E_{\frac{a}{s^r}\alpha_{ij}}E_{s^{N_1}\alpha_{kl}}E_{\frac{a}{s^r}\alpha_{ij}}^{-1}
& = \left[E_{\frac{a}{s^r}\alpha_{ij}},E_{s^{N_1}\alpha_{kl}}\right]E_{s^{N_1}\alpha_{kl}} \\
& = \left[E_{as^{p'}\alpha_{ij}},E_{s^{q'}\alpha_{kl}}\right]E_{s^{N_1}\alpha_{kl}}\;
\left(\textnormal{ by Corollary~\ref{c07}(i)}\right)\\
& = \textstyle{\prod_{t=1}^\nu E\left(s^{d_t}x_tZ_{p_tq_t}\right)} \;\textnormal{for}\;  d_t>0 \;
\textnormal{and}\; \nu\leq 5.
\end{align*}

\noindent This equation holds for any positive integers $p',q'$ with
$p'+q' = N_1-r$.
\begin{align*}
 \textnormal{ (ii) } E_{\frac{a}{s^r}\alpha_{ij}}
\left[E_{s^{N_2}x\beta_{il}^*},E_{s^{N_3}\gamma_{pq}^*}\right]
E_{\frac{a}{s^r}\alpha_{ij}}^{-1} & = \left[E_{\frac{a}{s^r}\alpha_{ij}}
\left[E_{s^{N_2}x\beta_{il}^*}, E_{s^{N_3}\gamma_{pq}^*}\right]\right]
\left[E_{s^{N_2}x\beta_{il}^*},E_{s^{N_3}\gamma_{pq}^*}\right]\\
& =\left[E_{s^{p''}\alpha_{ij}},\left[E_{s^{q''}x\beta_{il}^*},E_{s^{r''}\gamma_{pq}^*}\right]\right]\left[E_{s^{N_2}x\beta_{il}^*},E_{s^{N_3}\gamma_{pq}^*}\right]\\
&\hspace{5cm}\left(\textnormal{ by Corollary~\ref{c08}~(ii)}\right) \\
& =\textstyle{\prod_{t=1}^\nu E\left(s^{d_t}x_tZ_{p_tq_t}\right)} \;\textnormal{for}\;  d_t>0 \;
\textnormal{and}\; \nu\leq 14.
\end{align*}

\smallskip

 \noindent This equation holds for any positive integers $p'',q''$ and
$r''$ with $2p''+q''+r'' = N_2+N_3-2r$.  

\medskip

(iii)
$E_{\frac{a}{s^r}\alpha_{ij}}\left[E_{s^dx\frac{\beta_{il}^*}{2}},
E_{s^{N_1}\alpha_{kl}}\right]E_{\frac{a}{s^r}\alpha_{ij}}^{-1} =
\left[E_{\frac{a}{s^r}\alpha_{ij}},\left[E_{s^dx\frac{\beta_{il}^*}{2}},
E_{s^{N_1}\alpha_{kl}}\right]\right]
\left[E_{s^dx\frac{\beta_{il}^*}{2}},E_{s^{N_1}\alpha_{kl}}\right]$

\vspace*{1.5mm} \hspace*{6cm}
$=\left[E_{s^{p'''}\alpha_{ij}},\left[E_{s^{q'''}x\frac{\beta_{il}^*}{2}},E_{s^{r'''}\alpha_{kl}}\right]\right]\left[E_{s^dx\frac{\beta_{il}^*}{2}},E_{s^{N_1}\alpha_{kl}}\right]$

\vspace*{2mm} \hspace*{11cm} ( by Corollary~\ref{c08}~(i))

\vspace*{1mm} \hspace*{6cm} ${=\prod_{t=1}^\nu
E\left(s^{d_t}x_tZ_{p_tq_t}\right)}$for  $d_t > 0$ and $\nu\leq 14.$
 
 \medskip
 
 \noindent This equation holds for any positive integers $p''',q'''$
and $r'''$ with $2p'''+q'''+r''' = N_1+d-2r$.

\smallskip

\noindent Hence Equation~\eqref{e02} is of the form
${\prod_{t=1}^\nu E(s^{d_t}x_tZ_{p_tq_t})}
\;\textnormal{for}\;  d_t>0 \; \textnormal{and}\; \nu\leq 52.$

\end{proof}

\begin{lemma}\textbf{{\em (Dilation Lemma)}}\label{l10} Let $A$ be a
commutative ring and $Q,P$ be free modules of rank $n$ and $m$
respectively. Let $s$ be a non-nilpotent element of $A$ and $M = Q
\perp H(P).$ Let $\theta(X) \in~\O_{A[X]}(M[X])$ with $\theta(0)=Id.$
Let $Y,Z \in \hom(Q,P)$ or $\hom(Q,P^*)$. If
$\theta_s(X)=~(\theta(X))_s \in\EO_{{A[X]}_s}\left({M[X]}_s\right)$, then
for $d \gg 0$ and for all $b \in (s)^d A$, we have $\theta(bX) \in
\EO_{A[X]}\left(M[X]\right)$.
\end{lemma}
\begin{proof} Let $\theta_s(X) \in \EO_{A_s[X]}\left(M_s[X]\right)$. Then
$\theta_s(X) = \prod_{k=1}^{r} E\left(\alpha_{i_kj_k}(X)\right)$,
where 

\vspace{1mm} \noindent$\alpha_{i_{k}j_{k}}(X) \in \hom(Q_s[X],P_s[X])$
or  $\hom(Q_s[X],P_s^*[X])$ for all $k \in \mathbb{N}$, $i_k \in \{
1,2,...,m\}$ and 

\vspace{1mm}
\noindent$j_k \in \{ 1,2,...,n\}.$

Let $\alpha_{i_kj_k}(X) = \alpha_{i_kj_k}(0)+
X{\alpha}'_{i_kj_k}(X)$. By the splitting property, we can write
\[ E\left(\alpha_{i_{k}j_{k}}(X)\right) =
E\left(\frac{\alpha_{i_{k}j_{k}}(0)}{2}\right)E\left(X{\alpha}'_{i_{k}j_{k}}(X)\right)E\left(\frac{\alpha_{i_kj_k}(0)}{2}\right).
\]
\[ \textnormal{Then} \;\theta_s(X) = \textstyle{\prod_{k=1}^{r+1}
E\left(\frac{\alpha_{i_kj_k}(0)}{2}\right)E\left(X{\alpha}'_{i_{k}j_{k}}(X)\right)E\left(\frac{\alpha_{i_{k}j_{k}}(0)}{2}\right)}.
\] By Lemma~\ref{l08}, one has
\[ \theta_s(X) = \textstyle{\prod_{k=1}^{r+1} \gamma_k
E\left(X{\alpha}'_{i_kj_k}(X)\right)\gamma_k^{-1}},
\] where $\gamma_k = \prod_{j=1}^k a_j$ with 
		     \[
		     \begin{array}{llll}
		     \medskip
                     a_1 &=& E\left({\frac{\alpha_{i_1j_1}(0)}{2}}\right),\hphantom{E\left({\frac{\alpha_{i_kj_k}(0)}{2}}\right) fora}
			  a_{r+1} =E\left({\frac{\alpha_{i_rj_r}(0)}{2}}\right),\\
		      a_k &=&E\left({\frac{\alpha_{i_{k-1}j_{k-1}}(0)}{2}}\right)E\left({\frac{\alpha_{i_kj_k}(0)}{2}}\right)\quad\textnormal{for}\; k=2,...,r.&  \\ 
                     \end{array}
                     \]
Hence we can write
\begin{equation*} \theta_s(s^dX) = \textstyle{\prod_{k=1}^{r+1} \gamma_k
E\left(s^d X{\alpha}'_{i_kj_k}(s^d X)\right)\gamma_k^{-1}} \quad
\textnormal{for } d\gg 0.
\end{equation*}
\noindent \textbf{Claim :} If $\xi = \prod_{j=1}^k E\left(c_j\right),
c_j\in M_s,$ then for  $\xi E\left(s^d xZ_{ij}\right) {\xi}^{-1}$, we have
a product decomposition given by 
\begin{equation}\label{e00} \xi E\left(s^d xZ_{ij}\right) {\xi}^{-1} =
\textstyle{\prod_{t=1}^{\lambda_k} E\left(s^{d_t}x_tZ_{i_tj_t}\right)}
\end{equation} with $d_t\rightarrow \infty$ for $d \gg 0$, $x_t \in
A.$

\vspace{1mm}

\noindent {\it Proof of the Claim.}  We do this by induction on $k$. 

Let $\xi = \xi_1\xi_2 \ldots \xi_k$, where $\xi_i =
E\left(c_i\right)$. When $k=1$, by Lemma~\ref{l09}, we have a product
decomposition
\begin{equation*} \xi_1 E\left(s^d xZ_{ij}\right)\xi_1^{-1} =
\textstyle{\prod_{t=1}^{\lambda_1} E\left(s^{d_t}x_tZ_{i_tj_t}\right)}
\end{equation*} with $d_t\rightarrow \infty$ for $d\gg 0$. Now assume
that the result is true for $k-1$. i.e. we have
\begin{equation*} \xi_1\xi_2 \ldots \xi_{k-1} E\left(s^d
xZ_{ij}\right)\left(\xi_1\xi_2 \ldots \xi_{k-1}\right)^{-1} =
\textstyle{\prod_{t=1}^{\lambda_{k-1}} E\left(s^{d_t}x_tZ_{i_tj_t}\right)}
\end{equation*} with $d_t\rightarrow \infty$ for $d\gg 0$. Now by
Lemma~\ref{l09}, we can write
\begin{equation*} \xi_k E\left(s^d xZ_{ij}\right)\xi_k^{-1} =
\textstyle{\prod_{t=1}^{\lambda_{k-1}} E\left(s^{d_t}x_tZ_{i_tj_t}\right)}=
\mu_1\mu_2 \ldots \mu_\lambda \;\textnormal{(say).}
\end{equation*} Hence we have
\begin{align*} \left(\xi_1\xi_2 \ldots \xi_{k-1}\xi_k\right)
E\left(s^d xZ_{ij}\right) \left(\xi_1\xi_2 \ldots
\xi_{k-1}\right)^{-1} & = \left( \xi_1\xi_2 \ldots \xi_{k-1} \right)
\mu_1\mu_2 \ldots \mu_\lambda \left(\xi_1\xi_2 \ldots
\xi_{k-1}\right)^{-1}\\ &=\left(\xi_1\xi_2 \ldots
\xi_{k-1}\right)\mu_1(\xi_1\xi_2 \ldots \xi_{k-1})^{-1} (\xi_1\xi_2
\ldots \xi_{k-1})\\ &\hspace{6mm}\mu_2 (\xi_1\xi_2 \ldots
\xi_{k-1})^{-1}\ldots  (\xi_1\xi_2 \ldots \xi_{k-1})\\
&\hspace{6mm}\mu_\lambda \left(\xi_1\xi_2 \ldots
\xi_{k-1}\right)^{-1}.
\end{align*} Now applying induction to each of the expressions
$\xi_1\xi_2 \ldots \xi_{k-1}\mu_l\left(\xi_1\xi_2 \ldots
\xi_{k-1}\right)^{-1}$ as $l$ varies from 1 to $\lambda$, we have a
product decomposition as in Equation(\ref{e00}). Therefore we can write
\begin{equation*} \theta_s\left(s^d X\right) =
\textstyle{\prod_{k=1}^{r+1}\prod_{t=1}^{\lambda_k}
E\left(s^{d_t}x_tZ_{i_tj_t}\right)}
\end{equation*} for $d$ large enough. The terms $s^{d_t}x_t$ for $1 \leq t \leq \lambda_k$ is
contained in $M[X]$ as required. Hence
\begin{equation*} \theta(bX) =
\textstyle{\prod_{k=1}^{r+1}\prod_{t=1}^{\lambda_k}
E\left(s^{d_t}x_tZ_{i_tj_t}\right)} \in \EO_{A[X]}\left(M[X]\right) 
\end{equation*} for all $b\in(s)^d A.$
\end{proof}
% \begin{theorem}\textbf{\emph {(Local-Global Principle)}}\label{t00}
% Let $\theta(X) \in \O_A(M[X])$ with $\theta(0) = Id.$ If
% $\theta(X)_\mathfrak{m} \in EO_{A_{\mathfrak{m}}}(M_\mathfrak{m}[X])$
% for all maximal ideals $\mathfrak{m}$ of $A$, then $\theta(X) \in
% \EO_A(M[X])$.
% \end{theorem
 \begin{proof}[{Proof of the Theorem~\ref{t00}}] Let $\mathfrak{m}$ be
a maximal ideal of $A$. Choose an element $s_\mathfrak{m}$ from
$A\setminus \mathfrak{m}$ such that 
\[\theta(X)_{s_\mathfrak{m}}\in
EO_{{A_s[X]}_{\mathfrak{m}}}({M_s[X]}_{\mathfrak{m}}).\] Define
\begin{equation*} \kappa(X,Y) = \theta(X+Y)\theta(Y)^{-1}.
\end{equation*} Clearly $\kappa(X,Y)_{s_\mathfrak{m}} \in
EO_{{A_s[X,Y]}_{\mathfrak{m}}}({M_s[X]}_{\mathfrak{m}})$ and $\kappa(0,Y)=
Id$.

Now by applying Dilation Lemma with $A[Y]$ as the base ring, we get 
$$\kappa(b_{\mathfrak{m}}X,Y)\in \EO_{A[X,Y]}\left(M[X,Y]\right),$$ where $b_{\mathfrak{m}} \in
(s_{\mathfrak{m}}^N)$ for some $N \gg 0$.

Since $A$ is the ideal generated by $\{s_{\mathfrak{m}}\}_{\mathfrak{m} \in \max A},$ there
exists maximal ideals $\mathfrak{m}_1, \ldots, \mathfrak{m}_r$ and
elements $s_{\mathfrak{m}_i} \in A \setminus \mathfrak{m}_i$ such that
$A = \sum_{i=1}^r(s_{\mathfrak{m}_i})$. Therefore
$$A = \sum_{i=1}^r (s^{N_i}_{\mathfrak{m}_i})$$
for any $N_i > 0$. Hence for $b_{\mathfrak{m}_i} \in
(s^{N_i}_{\mathfrak{m}_i})$ with $N_i \gg 0$, there exists  elements $d_1,
\ldots, d_r \in A$ satisfying
$$\sum_{i=1}^r d_i b_{\mathfrak{m}_i} = 1.$$

Observe that $\kappa(d_ib_\mathfrak{m_i}X,Y) \in \EO_{A[X,Y]}(M[X,Y])$ for $1
\leq i \leq r$ .
\begin{align*} \theta(X) =&\theta\textstyle{\left(\sum_{i=1}^r d_i
b_{\mathfrak{m}_i} X\right)}\;\theta\left(\sum_{i=2}^r d_i
b_{\mathfrak{m}_i}X\right)^{-1} \theta\left(\sum_{i=2}^r d_i
b_{\mathfrak{m}_i}X\right)\;\theta\left(\sum_{i=3}^r d_i
b_{\mathfrak{m}_i}X\right)^{-1}\cdots\\ &\hspace{2mm}\theta
\left(d_{r-1} b_{\mathfrak{m}_{r-1}} X + d_r b_{{\mathfrak{m}_r}}X
\right)\; \theta \left( d_r b_{\mathfrak{m}_r}X \right)^{-1} \theta
\left( d_r b_{\mathfrak{m}_r}X \right)\\ =& \textstyle{\prod_{i=1}^{r-1}\kappa(d_i
b_{{\mathfrak{m}}_i}X,T_i)\kappa(d_rb_{{\mathfrak{m}}_r} X,0)},
\end{align*} where $T_i= \sum_{k=i+1}^r
d_{k}b_{\mathfrak{m}_{k}}X$. Hence $\theta(X) \in \EO_{A[X]}(M[X])$ .

\end{proof}

\subsection{A Local-Global principle for $\EO(Q\perp h^m)\cdot \O(h^m)$}

\begin{theorem}[{\cite{MR748229}, Theorem $2.5$}]\label{RARII} Let $A$
be a ring with generalized dimension $\ge d$. Let $(Q,q)$ be a
diagonalizable quadratic $A$-space. Consider the quadratic $A$-space
$Q\perp H(P)$, where rank\;$(P) >d$. Then
\begin{eqnarray*} 
\O_A(Q\perp H(P)) & = &\EO_A(Q, H(P))\cdot\O_A(H(P)) \\
& = &
\{\varepsilon \beta~|~ \varepsilon \in \EO_A(Q, H(P)), \beta \in \O_A(H(P))\}\\
& = & 
\{\beta \varepsilon ~|~ \varepsilon \in \EO_A(Q, H(P)), \beta \in \O_A(H(P))\}\\
&=& \O_A(H(P))\cdot\EO_A(Q, H(P)).
\end{eqnarray*}
\end{theorem}
\begin{lemma}\textbf{\em {(Dilation Lemma)}}\label{d2} Let $A$ be a commutative ring with generalized dimension $\ge d$ and $Q$ be a free module of rank $n$. Let $(Q, q)$ be a diagonalizable quadratic $A$-space. Let $s$ be a non-nilpotent element of $A$ and $m > d$. Let $\theta(X) \in \O_{A[X]}(Q\otimes A[X]\perp h^m)\cdot \O_{A[X]}(h^m)$ with $\theta(0)=Id.$ If $\theta_s(X)=(\theta(X))_s \in \EO_{{A[X]}_s}(Q\otimes {A[X]}_s \perp h^m)\cdot \O_{{A[X]}_s}(h^m),$ then for $d \gg 0$ and for all $b \in (s)^d A$, we have $\theta(bX) \in \EO_{A[X]}(Q\otimes A[X] \perp h^m)\cdot \O_{A[X]}(h^m)$.
\end{lemma}
\begin{proof} The proof is similar to Lemma~\ref{l10}. For, if $\theta_s(X) = \varepsilon(X) \beta(X)$ with \\$\varepsilon(X) \in \EO_{{A[X]}_s}(Q\otimes {A[X]}_s \perp h^m)$, $\beta(X) \in \O_{{A[X]}_s}(h^m)$, then $\theta(0) = I = \varepsilon(0)\beta(0);$ whence $\theta_s(X) = \{\varepsilon(X)\varepsilon(0)^{-1}\}\{\beta(0)^{-1} \beta(X)\}$.
In other words, we may assume at the onset that $\varepsilon(0) = Id$ and $\beta(0) = Id.$ The rest of the proof follows from Lemma~\ref{l10}.
\end{proof}
\begin{theorem}\textbf{\emph {(Local-Global Principle)}}\label{l2} 
 Let $A$ be a commutative ring with generalized dimension $\ge d$ and let $(Q, q)$ be a 
diagonalizable quadratic $A$-space. Assume that $Q$ is a free
module of rank $n$. Let $m > d$ and let $\theta(X) \in \O_{A[X]}(Q\otimes A[X]\perp h^m)$ with $\theta(0) =
Id.$ If $\forall\; \mathfrak{m} \in \max(A)$, $\alpha_{\mathfrak{m}}=
\beta_{\mathfrak{m}}\gamma_{\mathfrak{m}}$, where $\beta_{\mathfrak{m}} \in
\EO_{{A[X]}_{\mathfrak{m}}}\left((Q\otimes A[X])_{\mathfrak{m}}\perp h^m \right)$,
$\gamma_{\mathfrak{m}}\in \O_{{A[X]}_{\mathfrak{m}}}(h^m)$ with
$\beta(0) = Id, \gamma(0) = Id$. Then $\alpha = \beta\gamma$ with
$\beta \in \EO_{A[X]}((Q\otimes A[X])\perp h^m), \gamma \in
\O_{A[X]}(h^m)$.
\end{theorem}
 \begin{proof} The proof follows in similar lines as Theorem \ref{t00}
except for the following.  Let $\mathfrak{m}$  be a maximal ideal of
$A$. Choose an element $s_\mathfrak{m}$ from $A\setminus \mathfrak{m}$
such that \[\theta(X)_{s_\mathfrak{m}}\in
EO_{{A[X]}_{s_\mathfrak{m}}}\left(Q\otimes {A[X]}_{s_\mathfrak{m}}\perp
h^m \right)\O_{{A[X]}_{s_\mathfrak{m}}}(h^m).\] Define
\begin{equation*} \kappa(X,Y) = \theta(X+Y)\theta(Y)^{-1}.
\end{equation*} Then
\begin{equation}\label{eqk} \kappa(X,Y) =
\varepsilon_1\eta_1\varepsilon_2\eta_2  
\end{equation} for $\varepsilon_1, \varepsilon_2 \in
EO_{A[X,Y]}(Q\otimes A[X,Y]\perp h^m)$ and $\eta_1, \eta_2 \in
\O_{A[X,Y]}(h^m)$. Since 
\[\EO_{A[X,Y]}(Q\otimes A[X,Y]\perp h^m)\cdot \O_{A[X,Y]}(h^m) = \O_{A[X,Y]}(h^m)\cdot \EO_{A[X,Y]}(Q\otimes
A[X,Y]\perp h^m),\] by Theorem~\ref{RARII}, we can write Equation~\eqref{eqk} as
$$\kappa(X,Y) = \varepsilon_1\varepsilon'_2\eta'_1\eta_2$$ for some $\varepsilon'_2 \in \EO_{A[X,Y]}(Q\otimes A[X,Y]\perp h^m)$ and $\eta'_1 \in \O_{A[X,Y]}(h^m)$.
 
Then $\kappa(X,Y)_{s_\mathfrak{m}} \in
EO_{{A[X,Y]}_{s_\mathfrak{m}}}(Q\otimes {{A[X,Y]}_{s_\mathfrak{m}}} \perp
h^m)\cdot \O_{{A[X,Y]}_{s_\mathfrak{m}}}(h^m)$ and $\kappa(0,Y)= Id$.

Therefore, by applying Lemma~\ref{d2} with base ring
$A[Y]$, \[\kappa(b_\mathfrak{m}X,Y)\in \EO_{A[X,Y]}(Q\perp A[X,Y]
\perp h^m)\cdot \O_{A[X,Y]}(h^m),\]  where $b_\mathfrak{m} \in
(s_\mathfrak{m}^N)$ for some $N \gg 0$.

\end{proof}
\section{Extendability of Quadratic Spaces}

In this section, we apply the Local Global principle to prove the principal result [Theorem~\ref{Main}] on 
the extendability of quadratic $A[T]$-spaces of Witt index $\ge d$ over an equicharacteristic regular local
ring of dimension $d$. 

We begin with the following crucial observation.
\begin{lemma}\label{Pop} Let $A$ be a regular local ring containing a field. Let $(Q, q) \perp h$ be a quadratic $A[T]$-space.
If $(Q/TQ \perp h)$ is hyperbolic, then $(Q, q) \perp h$ is hyperbolic.
\end{lemma}
\begin{proof} In \cite{MR0818160}, D. Popescu showed that
if $A$ is a geometrically regular local ring (over a field $k$), or when
the characteristic of the residue field is a regular parameter in
$A$, then it is a filtered inductive limit of regular local rings
essentially of finite type over the integers (or over $k$).

In view of this, we may regard $(Q, q) \perp h$ to be a
quadratic $B[T]$-space over some regular local ring B essentially
of finite type over $k$ with $(Q/TQ, q/(T)) \perp h$ hyperbolic.
In view of Proposition~1.3 of \cite{MR727375},  $(Q, q) \perp h$ is hyperbolic over
$B[T]$, whence over $A[T]$.
\end{proof}
\begin{theorem}\label{Main} Let $(A,\mathfrak{m})$ be an equicharacteristic
regular local ring of dimension $d$ and $2\in A^*$. Then every quadratic
$A[T]$-space $(Q, q)\perp h^n$ with Witt index $n\geq d$ is extended from
$A$.
\end{theorem}
\begin{proof} 

Let $\{{{\pi}_1, {\pi}_2, \ldots, {\pi}_d}\}$ be a regular system of parameters
generating the maximal ideal $\mathfrak{m}$ of $A$.

Let $A^l$ denote the $(\pi_1, \ldots, \pi_l)$-adic completion of $A$. Observe
that $A^d$ is isomorphic to the power series ring $k[[X_1, \ldots, X_d]]$ by
Cohen structure theorem, where $k$ is the residue field $A/\mathfrak{m}$ of $A$.
Observe also that $A^l$ is the $(\pi_l)$-adic completion of $A^{l-1}$.

We now recall Amit Roy's garland of patching diagrams in \cite{MR0656854}:

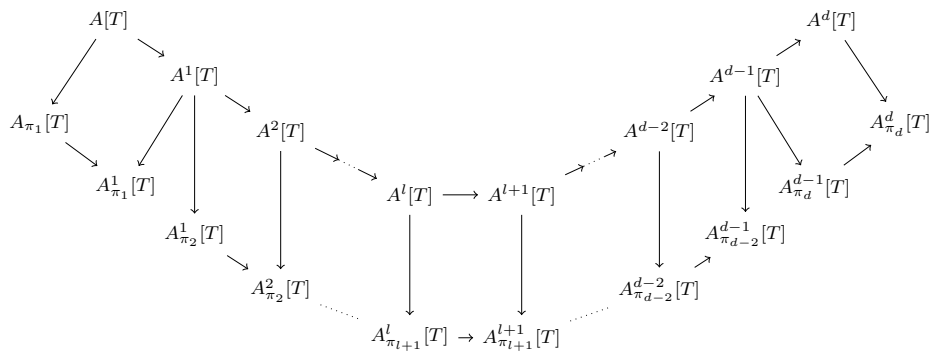
  \begin{center}
\scalebox{0.88}{
\begin{tikzpicture}[xscale=1.3,yscale=-1.2, font=\scriptsize, swap]
%top row 

%aligned at line 1
  \node (A0_1) at (1, -0.7) {$A[T]$};
  \node (A0_10) at (9.4, -0.7) {$A^d[T]$};
%aligned at line 2
  \node (A1_2) at (2, 0) {$A^1[T]$};  
  \node (A1_9) at (8.4, 0) {$A^{d-1}[T]$};
%aligned at line 3
  \node (A2_3) at (3, 0.7) {$A^2[T]$};
  \node (A2_8) at (7.4, 0.7) {$A^{d-2}[T]$};
%aligned at line 4
  \node (A4_5) at (4.5, 1.5) {$A^l[T]$};
  \node (A4_6) at (5.8, 1.5) {$A^{l+1}[T]$};

%top arrows 

  \path (A0_1) edge [->]node [auto] {$\scriptstyle{}$} (A1_2);

  \node (A1_0) at (0.2, 0.6) {$A_{\pi_1}[T]$};
  \node (A1_11) at (10.2, 0.6) {$A^d_{\pi_d}[T]$};
  \node (A2_1) at (1.2, 1.4) {$A^1_{\pi_1}[T]$};
  \node (A2_2) at (2, 2.0) {$A^1_{\pi_2}[T]$};
  \node (A2_9) at (8.4, 2) {$A^{d-1}_{\pi_{d-2}}[T]$};
  \node (A2_10) at (9.2, 1.4) {$A^{d-1}_{\pi_d}[T]$};
  \node (A3_3) at (3, 2.7) {$A^2_{\pi_2}[T]$};
  \node (A3_8) at (7.4, 2.7) {$A^{d-2}_{\pi_{d-2}}[T]$};
  \node (A5_5) at (4.5, 3.3) {$A^l_{\pi_{l+1}}[T]$};
  \node (A5_6) at (5.8, 3.3) {$A^{l+1}_{\pi_{l+1}}[T]$};
  
  \path (A3_8) edge [->]node [auto] {$\scriptstyle{}$} (A2_9);
  \path (A0_1) edge [->]node [auto] {$\scriptstyle{}$} (A1_0);
  \path (A4_6) edge [->]                 ($(A4_6)!.45!(A2_8)$)
               ($(A4_6)!.45!(A2_8)$) edge[dotted]  ($(A4_6)!.60!(A2_8)$)
                ($(A4_6)!.60!(A2_8)$) edge[->]  (A2_8);
  \path (A2_3) edge [->]                 ($(A2_3)!.45!(A4_5)$) 
               ($(A2_3)!.45!(A4_5)$) edge[dotted] ($(A2_3)!.60!(A4_5)$)
               ($(A2_3)!.60!(A4_5)$) edge[->]     (A4_5);
  \path (A1_2) edge [->]node [auto] {$\scriptstyle{}$} (A2_2);
  \path (A0_10) edge [->]node [auto] {$\scriptstyle{}$} (A1_11);
  \path (A1_2) edge [->]node [auto] {$\scriptstyle{}$} (A2_1);
  \path (A2_8) edge [->]node [auto] {$\scriptstyle{}$} (A3_8);
  \path (A2_3) edge [->]node [auto] {$\scriptstyle{}$} (A3_3);
  \path (A4_5) edge [->]node [auto] {$\scriptstyle{}$} (A4_6);
  \path (A5_5) edge [->]node [auto] {$\scriptstyle{}$} (A5_6);
  \path (A2_10) edge [->]node [auto] {$\scriptstyle{}$} (A1_11);

  \path (A1_9) edge [->]node [auto] {$\scriptstyle{}$} (A0_10);
  \path (A1_9) edge [->]node [auto] {$\scriptstyle{}$} (A2_10);
  \path (A1_0) edge [->]node [auto] {$\scriptstyle{}$} (A2_1);
  \path (A2_2) edge [->]node [auto] {$\scriptstyle{}$} (A3_3);
  \path (A1_2) edge [->]node [auto] {$\scriptstyle{}$} (A2_3);
  \path (A2_8) edge [->]node [auto] {$\scriptstyle{}$} (A1_9);
  \path (A1_9) edge [->]node [auto] {$\scriptstyle{}$} (A2_9);
  \path (A4_5) edge [->]node [auto] {$\scriptstyle{}$} (A5_5);
  \path (A4_6) edge [->]node [auto] {$\scriptstyle{}$} (A5_6);
  
  \path (A3_3) edge [-,dotted] (A5_5);
  \path (A5_6) edge [-,dotted] (A3_8);
\end{tikzpicture}
}
\end{center}

We use it below. Let us concentrate on the patching square 
${\mathcal{P}}_l(A)[T]$:
 \[
  \begin{tikzpicture}[xscale=3,yscale=-1.5]
    \node (A0_0) at (0, 0) {$A^l[T]$};
    \node (A0_1) at (1, 0) {$A^{l+1}[T]$};
    \node (A1_0) at (0, 1) {$(A^l)_{\pi_{l+1}}[T]$};
    \node (A1_1) at (1, 1) {$(A^{l+1})_{\pi_{l+1}}[T]$};
    \path (A0_0) edge [->]node [auto] {$\scriptstyle{}$} (A0_1);
    \path (A0_0) edge [->]node [auto] {$\scriptstyle{}$} (A1_0);
    \path (A0_1) edge [->]node [auto] {$\scriptstyle{}$} (A1_1);
    \path (A1_0) edge [->]node [auto] {$\scriptstyle{}$} (A1_1);
  \end{tikzpicture}
  \]

For all $l$, this is a cartesian square as rings. Moreover, by \cite{MR752647},
it is also a cartesian square of quadratic spaces. This will enable us to analyze the quadratic $A$-space.

We prove the result by induction on $d - l$, starting with $l = 0$. In
this case $A$ is a complete equicharacteristic regular local ring, whence
a power series ring over its residue field. We appeal to
\cite[{Theorem~1.1}]{MR727375}.

Assume the result for $d-l = m$. For $d-(l + 1) = m - 1$,
consider the patching diagram  ${\mathcal{P}}_{m-1}(A)[T]$.

We fix some notations as folows:

For a regular parameter $\pi$ of $A$, let ${Q}^l = Q\otimes{A^l}[T]$, $Q^0 = Q$, ${Q^l}_{{\pi}} = Q\otimes {A^l}_{{\pi}}[T]$ and for a quadratic $A$-space $Q_1$, we denote  ${Q_1}\otimes{A^l}$ by ${Q_1}^l$.

Let $(Q \perp h^n )/(T) = Q_1 \perp h^n $, where $Q_1$ is the quadratic $A$-space $Q/(T)$. Since $A^{m-1}$ is local, ${Q_1}^{m-1}$ is diagonalizable \cite[{Proposition~3.4}]{MR0491773}. Since $A^{m-1}$ is regular, by Karoubi's theorem \cite[{Chapter~VII, Theorem~2.1}]{MR2235330}, $(Q\perp h^n)^{m-1}$ is stably extended from $A^{m-1}$. 
Let $$\left( Q\perp {h^n} \right)^{m-1} \perp {h^r} \stackrel{ \simeq}{\longrightarrow}
A^{m-1}[T] \otimes \left({Q_1}^{m-1}\perp {h^{n +r}} \right), n \ge d.$$ 
Then $$\left(\left( Q\perp {h^n} \right)^{m-1} \perp {h^r}\right)_{\pi_m} \stackrel{ \simeq}{\longrightarrow}
\left( {\left(A^{m-1}\right)}_{\pi_m}[T] \otimes \left(\left({Q_1}^{m-1}\right)_{\pi_m}\perp {h^{n +r}} \right)\right), n \ge d.$$
By
\cite[{Theorem~3.3}]{MR748229}, we get the isomorphism
$$\left(\left( Q\perp {h^n} \right)^{m-1} \right)_{\pi_m} \stackrel{ \sigma}{\longrightarrow}
\left( {\left(A^{m-1}\right)}_{\pi_m}[T] \otimes \left(\left({Q_1}^{m-1}\right)_{\pi_m}\perp {h^{n}} \right)\right).$$

\noindent Using the extendability for quadratic spaces over $A^m[T]$ via induction hypothesis, we have
$$\tau: {{\left(Q\perp h^n \right)}^m}\stackrel{ \simeq}{\longrightarrow} {A^m[T] \otimes {\left({{Q_1}^m} \perp {h^n} \right)}}.$$ 
\noindent Now by identifying the quadratic spaces $\left( \left(\left( Q\perp {h^n} \right)^{m-1} \right)_{\pi_m} \bigotimes_{{\left(A^{m-1}\right)}_{\pi_m}[T]}  \left({A^{m}}_{\pi_m}[T]\right)	 \right) $ and 
\\${\left( \left( Q\perp {h^n} \right)^{m-1} \bigotimes_{A^{m-1}[T]} A^m[T]\right)}_{\pi_m} $ with $\left( \left( Q\perp {h^n} \right)^{m-1} \bigotimes_{A^{m-1}[T]} \left({(A^{m})}_{\pi_m}[T]\right) \right)$, via the patching for quadratic spaces from \cite{MR752647}, we have maps $\widetilde{\sigma}$, $\widetilde{\tau}$ corresponding to $\sigma$, $\tau$ and
\[\widetilde{\sigma}{\widetilde{\tau}}^{-1}\in
 \O_{(A^m)_{\pi_m}[T]}\left({{(Q_1 \perp {h^n})}^m}_{{{\pi}_m}} \right).\]
Since ${\left((A^m)_{\pi_m}\right)}_{\mathfrak{m}}$ is local,
${\left({{(Q_1)}^m}_{{\pi}_m}\right)}_{\mathfrak{m}}$ is diagonalizable and hence, by
Theorem~\ref{RARII}, 

\begin{equation*}
\O \left({\left({({Q_1}^m)}_{{\pi}_m} \right)}_{\mathfrak{m}} \perp h^n \right)  =  \EO \left({\left({({Q_1}^m)}_{{\pi}_m} \right)}_{\mathfrak{m}} \perp h^n \right) \cdot \O \left(h^n \right).
\end{equation*}

\noindent Therefore we can write 
\[\left({\widetilde{\sigma}{\widetilde{\tau}}^{-1}}\right)_{\mathfrak{m}}  =  \alpha_\mathfrak{m}\beta_\mathfrak{m},\]
where $\alpha_{\mathfrak{m}} \in   \EO_{{\left((A^m)_{\pi_m}\right)_{\mathfrak{m}}}[T]} \left({\left({({Q_1}^m)}_{{\pi}_m} \right)}_{\mathfrak{m}} \perp h^n \right)$ for some $\alpha \in \O_{(A^m)_{\pi_m}[T]} \left({\left({({Q_1}^m)}_{{\pi}_m} \right)} \perp h^n \right)$  with $\alpha(0) = Id$ and  $\beta_{\mathfrak{m}} \in \O_{{\left((A^m)_{\pi_m}\right)_{\mathfrak{m}}}[T]} \left(h^n \right)$ for some $\beta \in \O_{(A^m)_{\pi_m}[T]} \left(h^n \right)$ with $\beta(0) = Id$, via the same argument as in Lemma~\ref{d2}.

Then, by Theorem~\ref{l2}, we have \[\widetilde{\sigma}{\widetilde{\tau}}^{-1} = \alpha\beta\]
with $\alpha \in \O \left({\left({({Q_1}^m)}_{{\pi}_m} \right)} \perp h^n \right)$, $\alpha(0) = Id$, $\beta \in \O_{(A^m)_{\pi_m}[T]} \left(h^n \right)$ and $\beta(0) = Id$. Now via the \textquoteleft deep splitting\textquoteright \,technique introduced in \cite{MR727375}, we can write $\widetilde{\sigma}{\widetilde{\tau}}^{-1} = \beta \in \O(h^n)$.

Now we have
\begin{eqnarray*}
 \left(Q \perp h^n \right)^{m-1} &\simeq &\left(\left(\left( Q\perp {h^n} \right)^{m-1} \right)_{\pi_m}, Id, {{\left(Q\perp h^n \right)}^m} \right)  \\
& \simeq &{\left( {\left(A^{m-1}\right)}_{\pi_m}[T] \otimes \left(\left({Q_1}^{m-1}\right)_{\pi_m}\perp {h^{n}} \right), \alpha\beta, {A^m[T] \otimes {\left({{Q_1}^m} \perp {h^n} \right)}} \right)}\\
& \simeq &{\left({{Q_1}^{m-1}}_{{\pi}_m}[T]\perp h^n, \beta,
{{Q_1}^m}[T]\perp h^n\right)}\\
& \simeq &{{Q_1}^{m-1}[T]\perp \left(h^n,\beta, h^n \right)} =
{Q_1}^{m-1}[T] \perp Q_2,
\end{eqnarray*}
where $Q_2$ is the quadratic ${A}^{m-1}[T]$-space defined
by the patching technique. Now $${Q_1}^{m-1}[T] \perp Q_2 \perp h^r  \simeq Q^{m-1} \perp h^r \simeq {Q_1}^{m-1}[T] \perp h^{n+r}.$$ 

By cancellation of quadratic spaces over local rings \cite{MR0231844}, we have 
$Q_2 \perp h \simeq h^{n+1}.$ Since $\beta(0) = Id$, $Q_2/(T) \simeq h^n$. 
Thus, by Lemma~\ref{Pop}, $Q_2$ is extended from $A^{m-1}$, whence so is 
$(Q\perp h^n)^{m-1}$. Hence the result is true for $l+1$. Then the theorem 
follows by induction.
\end{proof}

\noindent
{\it Acknowledgements.} The first named author is indebted to her advisor 
B. Sury for his support, unstinting help and encouragement during the 
course of this work. She is also thankful to the Tata Institute of 
Fundamental Research, Mumbai for their hospitality from time to time as this work 
progressed. 

\setlength{\bibsep}{0.3pt}
\bibliographystyle{elsart-num-sort}
%\bibliography{../Articles,../Book}

\end{document}